\documentclass[a4,psfig,epsf,11pt]{article}
\usepackage[a4paper,left=2.6cm,right=2.6cm,top=2.3cm,bottom=1.9cm]{geometry}

\usepackage{setspace}
\usepackage{graphicx}
\usepackage[utf8]{inputenc}
\usepackage{amsfonts}
\usepackage{amsmath}
\usepackage{color}
\definecolor{verdemar}{rgb}{0.75,0.78,0.65}
\usepackage[mathcal]{euscript}
\usepackage{url}
\usepackage{subfigure}
\usepackage{psfrag}

\newtheorem{proposicao}{Proposition}[subsection]
\newtheorem{teorema}{Theorem}[subsection]
\newtheorem{lema}{Lemma}[subsection]
\newtheorem{corolario}{Corollary}[subsection]

\newtheorem{obs}{Remark}[subsection]
\newtheorem{Def}{Definition}[subsection]

\newtheorem{definicao}{Definition}[subsection]
\newenvironment{proof}{{\noindent\bf Proof.} }
                       {\hfill\rule{2.1mm}{2.1mm} \bigskip }

\newcommand{\R}{\mathbb{R}}

\begin{document}
\title{A Scalarization Proximal Point Method for Quasiconvex  Multiobjective Minimization \footnote{
This research was conducted with partial financial support from CAPES, through the Interagency Doctoral Program New Frontiers UFRJ/UFT.}}
\author{H. C. F. Apolinario\thanks{Federal  University of Tocantins, Undergraduate Computation Sciences Course, ALC NO 14 (109 Norte) AV.NS.15 S/N , CEP 77001-090, Tel: +55 63 8481-5168; +55 63 3232-8027; FAX +55 63 3232-8020, Palmas, Brazil (hellenato@cos.ufrj.br).}
\and{E. A. Papa Quiroz\thanks{ Mayor de San Marcos National University, Department of Ciencias Matem\'{a}ticas,  Lima, Per\'{u}. Federal University of Rio de Janeiro, Computing and Systems Engineering Department, post office box  68511,CEP 21945-970, Rio de Janeiro, Brazil(erikpapa@gmail.com).  The research of the second author was supported by the Postdoctoral Scholarship CAPES-FAPERJ Edital PAPD-2011}}\\
\and{P. R. Oliveira\thanks{Federal University of Rio de Janeiro, Computing and Systems Engineering Department, post office box  68511,CEP 21945-970, Rio de Janeiro, Brazil(poliveir@cos.ufrj.br).}}}
\date{\today}
\maketitle
\ \\[-0.5cm]

\begin{center}
{\bf Abstract}
\end{center}

In this paper we propose a scalarization proximal point method to solve multiobjective unconstrained minimization problems with locally Lipschitz and quasiconvex vector functions. We prove, under natural assumptions, that the sequence  generated by the method is well defined   and converges globally to a Pareto-Clarke critical point. Our method may be seen as an extension, for the non convex case, of the inexact proximal method for multiobjective convex minimization problems studied by Bonnel et al. (SIAM Journal on Optimization 15, 4, 953-970, 2005).

\section{Introduction}
\noindent

In this work we consider the unconstrained multiobjective minimization problem:
\begin{eqnarray}
\textrm{min}\lbrace F(x): x \in \mathbb{R}^n\rbrace
\label{prob}
\end{eqnarray}
where $F: \mathbb{R}^n\longrightarrow \mathbb{R}^m$ is a locally Lipschitz and
quasiconvex vector function on the Euclidean space $ \mathbb{R}^n.$  A motivation to study this problem are the consumer demand theory in economy, where the quasiconvexity of the objective vector function is a natural condition associated to diversification of the consumption, see Mas Colell et al. \cite{Colell}, and the quasiconvex optimization models in location Theory, see \cite{Gromicho}. Another motivation are the extensions of well known methods in convex optimization to quasiconvex one, we mentioned the following works:
\begin{itemize}
\item Bello Cruz et al. \cite{Bello}, considered the projected gradient method for solving the problem of finding a Pareto optimum of a quasiconvex multiobjective function. They proved the convergence of the sequence
generated by the algorithm to a stationary point and  when the components of the multiobjective function are pseudoconvex, they obtained the convergence to a weak Pareto solution.
\item da Cruz Neto et al. \cite{daCruzMult},  extended  the classical subgradient method for real-valued mi-nimization to multiobjective optimization. Assuming the basically componentwise quasiconvexity of the objective components they obtained the  full convergence of the sequence to a Pareto solution.
\item Papa Quiroz and Oliveira \cite{erikPaulo0601, Papa2011, Papa2}, have been extended the convergence of the proximal point method for quasiconvex minimization problems on general riemannian manifolds wich includes the euclidean space. Furthermore, in \cite{Papa1} the authors extended the convergence of the proximal point method for the nonnegative orthant.
\item Kiwiel \cite{Kiwiel}, extended the convergence of the subgradient method to solve quasiconvex minimization problems in Hilbert spaces.
\item Brito et al. \cite{Brito},  proposed an interior proximal algorithm inspired by
the logarithmic-quadratic proximal method for linearly constrained quasiconvex minimization problems. For that method, they
proved the global convergence when the proximal parameters go to zero. The latter assumption could be dropped when the function is assumed to be pseudoconvex.
\item Langenberg and Tichatschke \cite{Langenberg} studied the proximal method when the objective func-tion is quasiconvex and the problem is constrained to an arbitrary closed convex set and the regularization is a Bregman distance. Assuming that the function is locally Lipschitz and using the Clarke subdifferential, the authors proved the global convergence of the method to a critical point.
\end{itemize}

In this paper we are interested in extending the convergence properties of the proximal point method to solve the quasiconvex multiobjective problem (\ref{prob}). The proximal point method, introduced by Martinet \cite{Martinet}, to solve the problem  $\min\lbrace f(x) : x \in \mathbb{R}^n\rbrace$ where $f$ is a escalar function, generates a sequence $\lbrace x^k\rbrace_{k \in \mathbb{N}}\subset \mathbb{R}^n$, from an iterative process starting with a point $x^0 \in \mathbb{R}^n$, arbitrary, and $x^{k+1} \in \textrm{argmin}\lbrace f(x) + \frac{\lambda_k}{2}\Vert x - x^k \Vert ^2: x \in \mathbb{R}^n\rbrace$, where $\lambda_k > 0,$ is a regularization parameter.
It is well known, see Guler \cite{Guler1}, that if $f$ is convex  and
$\{\lambda_k\}$ satisfies $
\sum\limits_{k=1}^{+\infty}(1/\lambda_k)=+\infty,
$
then
$\lim_{k\rightarrow \infty}f(x^k)=\inf \{ f(x):x\in  \mathbb{R}^n  \}.$  Furthermore, if the optimal set is nonempty, we obtain that $\{x^k\}$ converges to an optimal solution of the problem.

When $F$ is convex in (\ref{prob}),  Bonnel at al. \cite{Iusem} have been proved the convergence of the  proximal point method for a weak Pareto solution of the problem (\ref{prob})  in a general context, see also  Villacorta and Oliveira \cite{Villacorta}  using proximal distances and  Gregório and Oliveira \cite{Gregorio} using a logarithmic quadratic proximal scalarization method.

In this work we introduce a scalarization proximal point method to solve the quasiconvex multiobjective minimization problem (\ref{prob}). The iteration is the following: given $x^{k} \in \mathbb{R}^n$, find $x^{k+1}\in \Omega_k$ such that:
 $$
 0 \in \partial^o\left( \left\langle F(.), z_k\right\rangle  + \dfrac{\alpha_k}{2}\left\langle e_k , z_k\right\rangle \Vert\ .\  - x^k \Vert ^2 \right) (x^{k+1}) + \mathcal{N}_{\Omega_k}(x^{k+1})
$$
where ${\partial}^o$ is the Clarke subdifferential, $\Omega_k= \left\{ x\in \mathbb{R}^n: F(x) \preceq F(x^k)\right\}$, $\alpha_k > 0 $, $\left\{e_k\right\}\subset \mathbb{R}^m_{++}$, $\left\|e_k\right\| = 1$, $\left\{z_k\right\} \subset \mathbb{R}^m_+\backslash \left\{0\right\}$, $\left\|z_k\right\| = 1$ and $\mathcal{N}_{\Omega_k}(x^{k+1})$ the normal cone to $\Omega_k$ at $x^{k+1}$.

We prove the  well  definition of  the sequence generated by the method and we obtain the global convergence to a Pareto-Clarke critical point and when $F$ is convex we obtain the convergence to a weak Pareto solution of the problem.

The paper is organized as follows:  In  Section 2 we recall some concepts and results basic on multiobjective optimization, quasiconvex and convex functions, Fr\'echet, Limiting and Clarke  subdiferential,  descent direction and Fej\'{e}r convergence theory. In Section 3 we introduce our method and analyze the convergence of the iterations. In Section 4, we present some quasiconvex optimization models and in Section 5 we give our conclusion and some ideas for future researchers.


\section{Preliminaries}

In this section, we present some basic concepts and results that are of fundamental importance for the development of our work. These facts can be found, for example, in Hadjisavvas \cite{Had}, Mordukhovich \cite{Mordukhovich} and, Rockafellar and Wets \cite{Rockafellar}.
\subsection{Definitions, notations and some basic results}

Along this paper  $ \mathbb{R}^n$ denotes an euclidean space, that is, a real vectorial space with the canonical inner product $\langle x,y\rangle=\sum\limits_{i=1}^{n} x_iy_i$ and the norm given by $||x||=\sqrt{\langle x, x\rangle }$.\\
Given a function {\small $f :\mathbb{R}^n\longrightarrow \mathbb{R}\cup\left\{+\infty\right\}$}, we
denote by {\small dom $(f)= \left\{x \in \mathbb{R}^n: f(x) < + \infty \right\}$}, the {\it effective domain } of $f$.  If {\small dom $(f)\neq \emptyset$}, $f $ is called proper.  If {\footnotesize $\lim  \limits_{\left\|x\right\|\rightarrow +\infty}f(x) = +\infty$}, $f$ is called coercive.  We denote by arg min $\left\{f(x): x \in \mathbb{R}^n\right\}$ the set of minimizer of the function $f$ and by  $f * $, the optimal value of problem: $\min \left\{f(x): x \in \mathbb{R}^n\right\},$  if it exists.
The \  function \ $f$ is {\it lower semicontinuous} at $\bar{x}$ if for all sequence $\left\{x_k\right\}_{k \in \mathbb{N}} $ such that $\lim  \limits_{k \rightarrow +\infty}x_k = \bar{x}$ we obtain that $f(\bar{x}) \leq \liminf \limits_{k \rightarrow +\infty}f(x_k)$.\\

\begin{Def}
Let $f:\mathbb{R}^n\longrightarrow \mathbb{R}\cup\left\{+\infty\right\}$ be a proper function. We say that $f$ is locally Lipschitz at $x\in$ {\small dom $(f)$} if there exists $\varepsilon_x>0$ such that $$\vert f(z) - f(y)\vert \leq L_{x}\Vert z - y\Vert, \
\forall z, y \in B(x,\varepsilon_x)\cap {\small dom (f)},$$ where $B(x,\varepsilon_x) = \lbrace y \in \mathbb{R}^n: \Vert y - x\Vert < \varepsilon_{x}\rbrace$ and $L_x$ is some positive number. $f$ is locally Lipschitz on $\mathbb{R}^n$ if $f$ is locally Lipschitz for each $x\in$ {\small dom $(f)$}
\end{Def}

The  next result ensures that the set of minimizers of a function, under some assumptions, is nonempty.
\begin{proposicao}{\bf (Rockafellar and Wets \cite{Rockafellar}, Theorem 1.9)}\\
Suppose that {\small $f:\mathbb{R}^n\longrightarrow \mathbb{R}\cup\left\{+\infty\right\}$} is {\it proper, lower semicontinuous} and coercive, then the optimal value  $ f^*$\ is finite and the set $\textnormal{arg min}$ $\left\{f(x): x \in \mathbb{R}^n\right\}$ is nonempty and compact.
\label{coercivaesemicont}
\end{proposicao}
\begin{Def}
Let $D \subset \mathbb{R}^n$ and $\bar{x} \in D$.   The normal cone at the point $\bar{x}$ related to the set $D$ is given by $\mathcal{N}_{D}(\bar{x}) = \left\{v \in \mathbb{R}^n: \langle  v,  x - \bar{x}\rangle \leq 0, \forall \ x \in D\right\}$.
\label{normal}
\end{Def}
\subsection{Multiobjective optimization}

In this subsection we present some properties and notation  on multiobjective optimization. Those basic facts can be seen, for example, in  Miettinen \cite{Kaisa} and Luc \cite{Luc}.\\
Throughout this paper we consider the cone $\mathbb{R}^m_+ = \{ y\in \mathbb{R}^m : y_i\geq0, \forall \  i = 1, ... , m \}$, which induce a partial order $\preceq$ in $\mathbb{R}^m$ given by, for $y,y'\in \mathbb{R}^m$,
$y\ \preceq\ y'$ if, and only if, $ y'\ - \ y$  $ \in \mathbb{R}^m_+$, this means that $ y_i \leq \ y'_i$ for all $ i= 1,2,...,m $ .   Given $ \mathbb{R}^m_{++}$ the above relation induce the following one $\prec$, induced  by the interior of this cone,  given by, $y\ \prec\ y'$, if, and only if, $ y'\ - \ y$  $ \in \mathbb{R}^m_{++}$, this means that $ y_i < \ y'_i$ for all $ i= 1,2,...,m$.  Those partial orders establish a class of problems known in the literature as Multiobjective  Optimization.\\ \\
Let us consider the unconstrained multiobjective optimization problem (MOP) :
\begin{eqnarray}
 \textrm{min} \left\{G(x): x \in \mathbb{R}^n \right\}
  \label{POM}
\end{eqnarray}
where $G:\mathbb{R}^n\longrightarrow \mathbb{R}^m$, with $G = \left(G_1, G_2, ... , G_m\right)^T$.
\begin{Def} {\bf (Miettinen \cite{Kaisa}, Definition 2.2.1)}
 A point $x^* \in \mathbb{R}^n$ is a Pareto solution of the problem $\left(\ref{POM}\right)$, if there does not exist $x \in  \mathbb{R}^n $ such that $ G_{i}(x) \leq G_{i}(x^*)$, for all $i \in \left\{1,...,m\right\}$ and $ G_{j}(x) <  G_{j}(x^*)$, for at least one index $ j \in \left\{1,...,m\right\}$ .
\end{Def}
\begin{Def}{\bf (Miettinen \cite{Kaisa},Definition 2.5.1)}
 A point $x^* \in \mathbb{R}^n$ is a weak Pareto solution of the problem $\left(\ref{POM}\right)$, if there does not exist $x \in  \mathbb{R}^n $ such that $ G_{i}(x) < G_{i}(x^*)$, for all $i \in \left\{1,...,m\right\}$.
\end{Def}
We denote by arg min$\left\{G(x):x\in \mathbb{R}^n \right\}$ and by arg min$_w$ $\left\{G(x):x\in \mathbb{R}^n \right\}$ the set of Pareto solutions and weak Pareto solutions to the problem $\left(\ref{POM}\right)$, respectively.  It is easy to check that\\ arg min$\left\{G(x):x\in \mathbb{R}^n \right\} \subset$ arg min$_w$ $\left\{G(x):x\in \mathbb{R}^n \right\}$.
\subsection{Quasiconvex and Convex Functions}
In this subsection we present the concept and characterization of quasiconvex functions and quasiconvex  multiobjective function. This theory can be found in Bazaraa et al. \cite{Bazaraa}, Luc \cite{Luc}, Mangasarian \cite{Mangasarian}, and their references.
\begin{Def}
Let $f:\mathbb{R}^n\longrightarrow \mathbb{R} \cup \{+ \infty \}$ be a proper function.  Then, f is
called quasiconvex if for all $x,y\in \mathbb{R}^n$, and for all $ t \in \left[0,1\right]$, it holds
that $f(tx + (1-t) y)\leq \textnormal{max}\left\{f(x),f(y)\right\}$.
\end{Def}
\begin{Def}
Let $f:\mathbb{R}^n\longrightarrow \mathbb{R} \cup \{+ \infty \}$ be a proper function.  Then, f is
called convex if for all $x,y\in \mathbb{R}^n$, and for all $ t \in \left[0,1\right]$, it holds
that $f(tx + (1-t) y)\leq tf(x) + (1 - t)f(y)$.
\end{Def}
Observe that if $f$ is a quasiconvex function then dom$(f)$ is a convex set. On the other hand, while a convex function can be characterized by the convexity of its epigraph, a quasiconvex function can
be characterized by the convexity of the lower level sets:

\begin{Def}{\bf (Luc \cite{Luc}, Corollary $6.6$)}
Let \ \ $F= (F_1,...,F_m)^T:\mathbb{R}^n\longrightarrow \mathbb{R}^m$ be a function, then $F$ is $\mathbb{R}^m_+$ -  quasiconvex if and only if every component function of $F$, $F_i: \mathbb{R}^n\longrightarrow \mathbb{R}$, is quasiconvex.
\end{Def}

\begin{Def}
Let \ \ $F= (F_1,...,F_m)^T:\mathbb{R}^n\longrightarrow \mathbb{R}^m$ be a function, then $F$ is $\mathbb{R}^m_+$ -  convex if and only if every component function of $F$, $F_i: \mathbb{R}^n\longrightarrow \mathbb{R}$, is convex.
\end{Def}

\begin{Def}
Let \ \ $F= (F_1,...,F_m)^T:\mathbb{R}^n\longrightarrow \mathbb{R}^m$ be a function, then $F$ is locally Lipschitz on $\mathbb{R}^n$ if and only if every component function of $F$, $F_i: \mathbb{R}^n\longrightarrow \mathbb{R}$, is locally Lipschitz on $\mathbb{R}^n$.
\end{Def}

\subsection{Fréchet  and Limiting Subdifferentials}
\begin{Def}
 Let $f: \mathbb{R}^n \rightarrow \mathbb{R} \cup \{ +\infty \}$ be a proper function.
 \begin{enumerate}
 \item [(a)]For each $x \in \textnormal{dom}(f)$, the set of regular subgradients (also called Fréchet subdifferential) of $f$ at $x$, denoted by $\hat{\partial}f(x)$, is the set of vectors $v \in \mathbb{R}^n$ such that
\begin{center}
$f(y) \geq f(x) + \left\langle v,y-x\right\rangle + o(\left\|y - x\right\|)$, where $\lim \limits_{y \rightarrow x}\frac{o(\left\|y - x\right\|)}{\left\|y - x\right\|} =0$.
\end{center}
Or  equivalently, $\hat{\partial}f (x) := \left\{ v \in \mathbb{R}^n : \liminf \limits_{y\neq x,\ y \rightarrow x} \dfrac{f(y)- f(x)- \langle v , y - x\rangle}{\lVert y - x \rVert} \geq 0 \right \}$.  If $x \notin \textnormal{dom}(f)$ then $\hat{\partial}f(x) = \emptyset$.
\item [(b)]The set of general subgradients (also called limiting subdifferential) $f$ at $x \in \mathbb{R}^n$, denoted by $\partial f(x)$, is defined as follows:
\begin{center}
$\partial f(x) := \left\{ v \in \mathbb{R}^n : \exists\  x_n \rightarrow x, \ \ f(x_n) \rightarrow f(x), \ \ v_n \in \hat{\partial} f(x_n)\  \textnormal{and}\  v_n \rightarrow v \right \}$.
\end{center}
\end{enumerate}
\label{fech}
\end{Def}
\begin{proposicao}
\label{Rockwets1}
 For a function  $f: \mathbb{R}^n \rightarrow \mathbb{R} \cup \{+ \infty \}$ and a point $\bar{x}\in \textnormal{dom}(f)$, the subgradient sets $\partial f(\bar{x})$ and $\hat{\partial} f(\bar{x})$ are closed, with $\hat{\partial} f(\bar{x})$ convex and $\hat{\partial}f\left(\bar{x}\right)$ $\subset$ $\partial f\left(\bar{x}\right)$.
\end{proposicao}
\begin{proof}
See Rockafellar and Wets \cite{Rockafellar}, Theorem 8.6.
\end{proof}

\begin{proposicao}{\bf (Fermat’s rule generalized)}
If a proper function $f: \mathbb{R}^n \rightarrow \mathbb{R} \cup \{+ \infty \}$ has a local minimum at $\bar{x} \in dom{(f)}$, then $0\in \hat{\partial} f\left(\bar{x}\right)$.
\label{otimo}
\end{proposicao}
\begin{proof}
See Rockafellar and Wets \cite{Rockafellar}, Theorem 10.1.
\end{proof}
\begin{proposicao}
Let $f, g : \mathbb{R}^n \rightarrow \mathbb{R} \cup \{+ \infty \}$ proper functions such that $f$ is locally Lipschitz at $\bar{x} \in$ dom$(f)\ \cap\ $dom$(g)$ and $g$ is lower semicontinuous function at this point. Then,
\begin{center}
$\partial(f + g)(\bar{x})\subset \partial f(\bar{x}) + \partial g(\bar{x})$
\end{center}
\label{mordukhovich}
\end{proposicao}
\begin{proof}
See Mordukhovich \cite{Mordukhovich} Theorem 2.33.
\end{proof}

\subsection{\bf Clarke Subdifferential}
\begin{definicao}
Let $f: \mathbb{R}^n \rightarrow \mathbb{R} \cup \{+ \infty \}$ be a proper locally Lipschitz function at $x \in \textnormal{dom}(f)$ and $d \in \mathbb{R}^n$.  The Clarke directional derivative of $f$ at $x$ in the direction $d$, denoted by $f^o(x,d)$, is defined as
\begin{center}
$f^o (x,d) = \limsup\limits_{t\downarrow 0 \ \  y \rightarrow x} \dfrac{f(y + td)- f(y)}{t}$
\end{center}
and the Clarke subdifferential of $f$ at $x$, denoted by $\partial ^of(x)$, is defined as
\begin{center}
$\partial ^of(x)= \lbrace w \in \mathbb{R}^n:\langle w , d \rangle \leq f^o (x,d), \forall\  d \in \mathbb{R}^n \rbrace$.
\end{center}
\label{subclarke1}
\end{definicao}
\begin{obs}
From the above definitions it follows directly that for all $x \in \mathbb{R}^n$, one has
$\hat\partial f(x) \subset \partial f(x) \subset \partial ^o f(x)$ (see Bolte et al. \cite{Bolte}, Inclusion (7)).
\label{frechetclarke}
\end{obs}

\begin{lema}
Let $f, g : \mathbb{R}^n \rightarrow \mathbb{R} \cup \{+ \infty \}$ be locally  Lipschitz functions at $x \in \mathbb{R}^n $.  Then, $\forall  d  \in \mathbb{R}^n$:
\begin{enumerate}
\item [(i)]$\left( f + g \right)  ^ o \left( x , d\right)   \leq f^o \left( x , d\right)   + g^o \left( x, d\right)$ ;
\item[(ii)]$ \left(\lambda f \right)^ o \left( x , d\right) = \lambda \left( f^o (x , d) \right),\ \forall \lambda \geq 0 $;
\item[(iii)]$ f ^ o \left( x , \lambda d\right) = \lambda f^o (x , d),\ \forall \lambda \geq 0 $.
\end{enumerate}
\label{algebra}
\end{lema}
\begin{proof}
It is immediate from Clarke directional derivative.
\end{proof}
\begin{lema}
 Let $f: \mathbb{R}^n \rightarrow \mathbb{R}$ be locally  Lipschitz function at $x$ and any scalar $\lambda$ , then
\begin{center}
$\partial ^o \left(\lambda f\right)(x)\subset \lambda \partial ^o f(x)$
\end{center}
\label{escalar}
\end{lema}
\begin{proof}
See Clarke \cite{Clarke}, Proposition 2.3.1.
\end{proof}

\begin{lema}
 Let $f_i: \mathbb{R}^n \rightarrow \mathbb{R} \cup \{+ \infty \}, i=1,2,...,m,$ be locally  Lipschitz functions at $x$, then
\begin{eqnarray*}
\partial ^o \left(\displaystyle\sum_{i=1}^{m} f_i\right)(x)\subset \displaystyle\sum_{i=1}^{m}\partial ^o f_i  (x)
\end{eqnarray*}
\label{subcontido}
\end{lema}
\begin{proof}
See Clarke \cite{Clarke}, Proposition 2.3.3.
\end{proof}

\begin{proposicao}
Let $f: \mathbb{R}^n \rightarrow \mathbb{R} \cup \{+ \infty \}$ be a proper locally Lipschitz function on $\mathbb{R}^n$. Then,  $f^ o $ is upper semicontinuous, i.e, if $ \lbrace(x^k, d^k)\rbrace$ is a sequence in $\mathbb{R}^n\times \mathbb{R}^n$ such that $\lim \limits_{k \rightarrow + \infty}(x^k, d^k) = (x , d) $ then
$\limsup \limits_{k \rightarrow + \infty}f^o (x^k, d^k) \leq f^o (x , d)$.
\label{limsup}
\end{proposicao}
\begin{proof}
See Clarke \cite{Clarke}, Proposition 2.1.1, (b).
\end{proof}

\begin{proposicao}
\label{t11}
Let $f:\mathbb{R}^n\longrightarrow \mathbb{R}$ be a quasiconvex locally Lipschitz function on $\mathbb{R}^n.$ If
$g  \in \partial^o f(x),$ such that $\left\langle g ,\tilde{x} - x \right\rangle > 0$ then, $f(x)\leq f(\tilde{x}).$
\end{proposicao}
\begin{proof}
See Aussel \cite{Aussel}, Theorem 2.1.
\end{proof}
\begin{proposicao}
Let $f : \mathbb{R}^n \longrightarrow \mathbb{R}$ be a convex function. Then $\partial^ o f(x)$ coincides with the subdifferential at $x$ in the sense of convex analysis, and $f^ o (x,d)$ coincides with the directional derivative $f'(x,d)$ for each $d$.
\label{igualdadfc}
\end{proposicao}
\begin{proof}
See Clarke \cite {Clarke}, Proposition 2.2.7
\end{proof}

\subsection{Descent direction}
We are now able to introduce the definition of Pareto-Clarke critical point for  locally Lipschitz functions on $\mathbb{R}^n$, which will play a key role in our paper.

\begin{Def}{\bf (Custódio et al. \cite{Custodio}, Definition $4.6$)}
Let $F= (F_1,...,F_m)^T:\mathbb{R}^n\longrightarrow \mathbb{R}^m $ be locally Lipschitz on $\mathbb{R}^n$. We say that $x^* \in \mathbb{R}^n$ is a Pareto-Clarke critical point of $F$ if,  for all directions $d \in \mathbb{R}^n$, there exists $i_0 = i_{0} (d) \in \lbrace1,...,m\rbrace$ such that $F^o_{i_o}(x^*, d)\geq 0$.
\label{paretoclarke}
\end{Def}
Definition $\ref{paretoclarke}$ says essentially that there is no direction in $\mathbb{R}^n$ that
is descent for all the objective functions (see, for instance, (Custódio et al. \cite{Custodio}). If a point is a Pareto minimizer (local or
global), then it is necessarily a Pareto-Clarke critical point .

\begin{obs}
\label{Descensoo}
Follows  from the previous definition that, if a point $x$ is not Pareto-Clarke critical, there exists a direction $d \in \mathbb{R}^n$ satisfying
\begin{center}
$F_{i}^o(x, d)  < 0, \forall \ i \in \left\{1,...,m\right\}$
\end{center}
This implies that, for each $i \in \lbrace 1,..., m \rbrace$, $d$ is a {\it descent direction}, for each function $F_i$, i.e, there exists \ \ $\varepsilon > 0 $, such that
\begin{center}
 $F_i(x + td)  < F_i(x), \forall \ t \in (0 , \varepsilon], \forall \ i \in \lbrace 1,..., m \rbrace $.
 \end{center}
It is a well known fact that such $d$ is a {\it descent
direction} for the multiobjective  function $F$ at $x$, i.e, $\exists \ \ \varepsilon > 0 $ such that
 \begin{center}
 $ F(x + td) \prec F(x), \ \forall \ t \in (0 , \varepsilon]$.
 \end{center}
 \end{obs}

\begin{proposicao}
Let $\bar{x}$ be a Pareto-Clarke critical point  of a locally Lipschitz $G:\mathbb{R}^n\longrightarrow \mathbb{R}^m.$  If $G$ is $\mathbb{R}^m_{+}$-convex, then $\bar{x}$ is weak Pareto solution of the problem (\ref{POM}).
\label{Corconvexlips}
\end{proposicao}
\begin{proof}
As $\bar{x}$ is a  Pareto-Clarke critical point  of $G$ then for all directions $d$ there exists $i_0 = i_{0} (d) \in \lbrace1,...,m\rbrace$ such that $G^o_{i_o}(\bar{x}, d)\geq 0.$ Now, due that $G$ is $\mathbb{R}^m_{+}-$ convex then the last is equivalent, see Proposition \ref{igualdadfc}, to
\begin{eqnarray}
G'_{i_o}(\bar{x}, d)\geq 0,
\label{naoparetoconvex}
\end{eqnarray}
where $G'_{i_0}(\bar{x}, d)$ is the directional derivative of the convex function $G_{i_0}$ at $\bar{x}$ in the direction $d$.\\
On the other hand, suppose by contradiction that $\bar{x}$ is not a weak Pareto solution of  the problem (\ref{POM}), then exists $x^* \in \mathbb{R}^n$ such that
$$
G(x^*) \prec G(\bar{x}),\  \textnormal{i.e},\  G_i(x^*) < G_i(\bar{x}), \forall i \in {1,...,m}.
$$
Thus, for all $i,$ there exists $\alpha=\alpha(i) > 0$ such that $G_{i}(x^*) =G_{i}(\bar{x}) - \alpha$.  Define  $x_\lambda = \lambda x^* + (1-\lambda)\bar{x}$, $\lambda \in (0,1)$.  From the  $\mathbb{R}^m_+$-convexity of $G$ we have
\begin{eqnarray*}
G_{i}(x_\lambda)=G_{i}(\lambda x^* + (1-\lambda)\bar{x})\leq \lambda G_{i}(x^*) + (1-\lambda)G_{i}(\bar{x})= -\alpha \lambda + G_{i}(\bar{x})
\end{eqnarray*}
Its follows that
\begin{center}
$\dfrac{G_{i}(\bar{x} + \lambda(x^* -\bar{x})) - G_{i}(\bar{x})}{\lambda} \leq - \alpha < 0$, $\forall \lambda \in (0,1)$.
\end{center}
Taking $\bar{d} = x^* - \bar{x} \in \mathbb{R}^n$ and  limit when $\lambda$ converges to zero in the above inequality we obtain a contradiction with (\ref{naoparetoconvex}). Therefore $\bar{x}$ is a weak Pareto solution of the problem  (\ref{POM}).
\end{proof}

\subsection{Fejér convergence}
\begin{Def}
A seguence $\left\{y_k\right\} \subset \mathbb{R}^n$ is said to be Fejér convergent to a set $U\subseteq \mathbb{R}^n$ if,
$\left\|y_{k+1} - u \right\|\leq\left\|y_k - u\right\|, \forall \ k \in \mathbb{N},\ \forall \ u \in U$.
\end{Def}
The following result on Fejér convergence is well known.
\begin{lema}
If $\left\{y_k\right\}\subset \mathbb{R}^n$ is Fejér convergent to some set $U\neq \emptyset$, then:
\begin{enumerate}
\item [(i)]The sequence $\left\{y_k\right\}$ is bounded.
\item [(ii)]If an accumulation point $y$ of $\left\{y_k\right\}$ belongs to $ U$, then $\lim  \limits_{k\rightarrow +\infty}y_k = y$.
 \end{enumerate}
\label{fejerlim1}
\end{lema}
\begin{proof}
See Schott \cite{Schott}, Theorem $2.7$.
\end{proof}


\section{Scalarization proximal point method (SPPM)}

We are interested in solving the unconstrained multiobjective optimization problem (MOP):
\begin{eqnarray}
\textrm{min}\lbrace F(x): x \in \mathbb{R}^n\rbrace
\label{pom}
\end{eqnarray}
where $F: \mathbb{R}^n \rightarrow \mathbb{R}^m$ is a vector function satisfying the following assumptions:
\begin{description}
\item [$\bf (H_1)$] $F$ is locally Lipschitz on $\mathbb{R}^n$.
\item [$\bf (H_2)$] $F$ is $\mathbb{R}^m_+$-quasiconvex.
\end{description}

 \subsection{The algorithm}
 In this subsection, we propose a Scalarization Proximal Point Method with quadratic regula-rization, denoted by {\bf SPPM}, to solve the problem $(\ref{pom})$.\\ \\
 {\bf SPPM Algorithm}
\begin{description}
\item [\bf Initialization:] Choose an arbitrary initial point
\begin{eqnarray}
x^0\in\mathbb{R}^n
\label{inicio}
\end{eqnarray}
\item [Main Steps:] Given $x^k,$ find $x^{k+1}\in \Omega_k$ such that
 \begin{eqnarray}
 0 \in \partial^o\left( \left\langle F(.), z_k\right\rangle  + \dfrac{\alpha_k}{2}\left\langle e_k , z_k\right\rangle \Vert\ .\  - x^k \Vert ^2 \right) (x^{k+1}) + \mathcal{N}_{\Omega_k}(x^{k+1})
 \label{subdiferencial}
 \end{eqnarray}
where $\Omega_k= \left\{ x\in \mathbb{R}^n: F(x) \preceq F(x^k)\right\}$, $\alpha_k > 0 $, $\left\{e_k\right\}\subset \mathbb{R}^m_{++}$, $\left\|e_k\right\| = 1$, $\left\{z_k\right\} \subset \mathbb{R}^m_+\backslash \left\{0\right\}$ and $\left\|z_k\right\| = 1$ .
\item [Stop Criterion:] If $x^{k+1}=x^{k}$ or $x^{k+1}$ is a Pareto-Clarke critical point, then stop. Otherwise, to do $k \leftarrow k+1$ and return to Main Steps.
\end{description}
\begin{obs} If $F$  is $\R^{n}_{+}-$convex  the main step  (\ref{subdiferencial})  is equivalent to:
\begin{eqnarray}
 x^{k+1}=\textrm{argmin} \left\{ \left\langle F(x), z_k\right\rangle+\frac{\alpha_k}{2}\left\langle e_k , z_k\right\rangle\left\|x - x^k\right\|^2 : x\in\Omega_k\right\}
 \label{subdiferencial2}
 \end{eqnarray}
This iteration has been studied by  Bonnel et al. \cite{Iusem}, so we can say that,  in certain sense, our iteration is an extension for the nonconvex case of that work . On the other hand, when $F$ is $\R^{n}_{+}-$quasiconvex, the regularized function  $F_k=\left\langle F(x), z_k\right\rangle + \frac{\alpha_k}{2}\left\langle e_k , z_k\right\rangle\left\|x - x^k\right\|^2$ is not necessarily quasiconvex and so  (\ref{subdiferencial2}) is a global optimization problem, it is the reason for which we consider the more weak  iteration  (\ref{subdiferencial}).

\end{obs}

\subsection{Existence of the iterates}

\begin{teorema}
 Let $F:\mathbb{R}^n\longrightarrow \mathbb{R}^m $ be a function satisfying $\bf (H_1), \bf (H_2)$  and $ 0 \prec F$ .  Then the sequence $\left\{x^k\right\}$, generated by the SPPM algorithm, given by $(\ref{inicio})$ and $(\ref{subdiferencial}),$ is well defined.
 \label{existe}
\end{teorema}
\begin{proof} We proceed by induction. It holds for $k=0,$ due to (\ref{inicio}). Assume that $x^k$ exists and define $\varphi_k(x)=\left\langle F(x), z_k\right\rangle + \frac{\alpha_k}{2}\left\langle e_k , z_k\right\rangle\left\|x - x^k\right\|^2 +\delta_{\Omega_k}(x)$, where $\delta_{\Omega_k}(.)$ is the indicator function of ${\Omega_k}$.  Then we have that min$\{\varphi_k(x): x \in \mathbb{R}^n\}$ is equivalent to min$\{\left\langle F(x), z_k\right\rangle + \frac{\alpha_k}{2}\left\langle e_k , z_k\right\rangle\left\|x - x^k\right\|^2: x \in \Omega_k\}$.  Due that $ 0 \prec F$ and $z_k \in \mathbb{R}^m_+\backslash \left\{0\right\}$ the function $\left\langle F(.), z_k\right\rangle$ is bounded from below. Then, by the lower boundedness and continuity of the function $\left\langle F(.), z_k\right\rangle$, as also, by the continuity and coercivity of  $||.-x^k||^2,$ and using Proposition \ref{coercivaesemicont}, we obtain that there exists $x^{k+1} \in \Omega_k$ which is a global minimum of $\varphi_k(.).$ From Proposition $\ref{otimo}$, $x^{k+1}$ satisfies
 $ 0 \in \hat{ \partial}\left( \left\langle F(.), z_k\right\rangle  + \dfrac{\alpha_k}{2}\left\langle e_k , z_k\right\rangle \Vert\ .\  - x^k \Vert ^2 + \delta_{\Omega_k}(.)\right) (x^{k+1})$ and by Proposition $\ref{Rockwets1}$ and Proposition $\ref{mordukhovich}$ , we have that
\begin{eqnarray}
0 \in  \partial\left( \left\langle F(.), z_k\right\rangle  + \dfrac{\alpha_k}{2}\left\langle e_k , z_k\right\rangle \Vert\ .\  - x^k \Vert ^2\right) (x^{k+1}) +  \mathcal{N}_{\Omega_k}(x^{k+1}).
\label{cone}
\end{eqnarray}
 From Remark $\ref{frechetclarke}$, the iteration $(\ref{subdiferencial})$ is obtained of $(\ref{cone})$.
\end{proof}
\begin{obs}{\bf (Huang and Yang \cite{Huang})} Without loss of generality, always we can assume that the function $F:\mathbb{R}^n\longrightarrow \mathbb{R}^m$ satisfies $0 \prec F.$ Of fact,  consider the following multiobjective optimization problem
\begin{center}
$(P^{'})\ \   \textnormal{min}\left\{e^{F(x)}:x\in\mathbb{R}^n\right\}$
\end{center}
Observe that both, (\ref{pom}) and $(P^{'})$, have the same set of Pareto solutions, weak Pareto solutions and Pareto-Clarke critical points. Furthermore, if $F$ is $\mathbb{R}^m_+$ - quasiconvex and locally Lipschitz on $\mathbb{R}^n,$ then $e^{F(x)}$ is also $\mathbb{R}^m_+$ - quasiconvex and locally Lipschitz on $\mathbb{R}^n$.  Therefore, along this paper and from now on we implicitly assume that $0 \prec F.$
\label{limitada}
\end{obs}
\begin{obs}
\label{interiornovacio}
We are interest in the asymptotic convergence of the (SPPM) algorithm, so we also assume along this paper that in each iteration $x^k$ is not a Pareto-Clarke critical point and $x^{k+1}\neq x^k.$ This implies, from Remark \ref{Descensoo} that the interior of $\Omega_{k+1},$ denoted by $\Omega_{k+1}^0,$ is nonempty.

When the condition $x^{k+1}=x^k$ is not satisfied, that is, if there exists $k_0$ such that $x^{k_0+1}=x^{k_0}$ then it is easy to prove that  this point is a Pareto-Clarke critical point of $F.$
\end{obs}
\subsection{Weak Convergence}
In this subsection we prove, under the assumption that the consecutive iterations converges to zero, that any cluster point is a Pareto-Clarke critical point of the problem $(\ref{pom}).$
\begin{proposicao}
 Let $F:\mathbb{R}^n\longrightarrow \mathbb{R}^m $ be a function satisfying $\bf (H_1)$ and $\bf (H_2).$  If $0 < \alpha_k < \tilde{\alpha}$, with $\tilde{\alpha}>0,$ and the sequence $\lbrace x^k\rbrace$  generated by the SPPM algorithm, $(\ref{inicio})$ and $(\ref{subdiferencial}),$ satisfies 
\begin{equation}
\label{succonsecutivas}
\lim_{k\rightarrow +\infty}||x^{k+1}-x^{k}||=0,
\end{equation}
and has a cluster point, then it is a Pareto-Clarke critical point of the problem $(\ref{pom})$.
\label{prop1}
\end{proposicao}
\begin{proof}
By assumption, there exists a convergent subsequence $\left\{x^{k_j}\right\}$ of $\left\{x^{k}\right\}$ whose limit is some $\widehat{x} \in \mathbb{R}^n$. Since $F$ is locally Lipschitz on $\mathbb{R}^n$, then the function  $\left\langle F(.), z\right\rangle$ is also locally Lipschitz on $\mathbb{R}^n$ and so, continuos for all $z \in \mathbb{R}^m$, in particular, for all $z \in \mathbb{R}^m_+ \backslash \left\{0\right\}$ , and
$\lim \limits_{j\rightarrow +\infty}\left\langle F(x^{k_j}) , z\right\rangle = \left\langle F(\widehat x),z\right\rangle$.  On the other hand, as $x^{k+1} \in \Omega_k$, we have $F(x^{k+1}) \preceq F(x^{k})$ and since $z \in \mathbb{R}^m_+ \backslash \left\{0\right\}$, we conclude that the sequence $\left\{\left\langle F(x^k),z\right\rangle\right\}$ is convergent to $\left\langle F(\widehat x),z\right\rangle$ because it is nonincreasing and admits a subsequence converging to $\left\langle F(\widehat x),z\right\rangle$. So {\small
$\lim \limits_{k\rightarrow +\infty}\left\langle F(x^{k}) , z\right\rangle = \left\langle F(\widehat {x}),z\right\rangle = inf_{k\in \mathbb{N}}\left\{\left\langle F(x^k),z\right\rangle\right\} \leq  \left\langle F(x^k),z\right\rangle$}. Thus,
$\left\langle F(x^k)-F(\widehat{x}),z\right\rangle \geq  0, \forall \ k \in \mathbb{N}, \forall \ z \in \mathbb{R}^m_+ \backslash \left\{0\right\}$. We conclude that $F(x^k) - F(\widehat{x}) \in \mathbb{R}^m_+$, i.e, $F(\widehat{x})\preceq F(x^k), \forall \ k  \in \mathbb{N}$. This implies that $\widehat{x} \in \Omega_{k}$.\\
Assume, to arrive at a contradiction, that $\widehat{x}$ is not Pareto-Clarke critical point in $\mathbb{R}^n$, then there exists a direction $d \in \mathbb{R}^n$ such that
\begin{eqnarray}
F_{i}^o(\widehat{x}, d)  < 0, \forall \ i \in \left\{1,...,m\right\}
\label{desce}
\end{eqnarray}
Therefore $d$ is a descent direction for the multiobjective function $F$ in $\widehat{x}$, so, $\exists \ \varepsilon > 0$ such that
$F(\widehat{x} + \lambda d) \prec F(\widehat{x}),\ \forall \ \lambda \in (0, \varepsilon].$ Thus, $ \widehat{x} + \lambda d \in \Omega_{k}$.\\
On the other hand,  as $\left\{x^{k}\right\}$ is generated by {\bf SPPM} algorithm, from Theorem $\ref{existe}$, $(\ref{subdiferencial})$,  Lemma $\ref{subcontido}$ and from Lemma $\ref{escalar}$, this implies that there exists $\beta_k (x^{k} - x^{k+1}) - v_k\in \partial ^ o \left( \left\langle F(.), z_k\right\rangle\right) (x^{k+1})$, with  $v_k \in \mathcal{N}_{\Omega_k}(x^{k+1})$ and
$\beta_k =\alpha_k\left\langle e_k , z_k\right\rangle > 0$, such that
\begin{equation}
\beta_k \langle x^k - x^{k+1} , p \rangle  - \langle v_k , p \rangle \leq \langle F(.), z_k \rangle ^ o (x^{k+1},p), \forall p\in \mathbb{R}^n
\label{subclarke4}
\end{equation}
Consider $p = (\widehat{x} + \lambda d) - x^{k+1}$ and as  $v_k \in \mathcal{N}_{\Omega_k}(x^{k+1})$, from $(\ref{subclarke4})$ we have
\begin{equation}
\beta_k \langle x^k - x^{k+1} , \widehat{x} + \lambda d - x^{k+1} \rangle  \leq \langle F(.), z_k \rangle ^ o (x^{k+1},\widehat{x} + \lambda d - x^{k+1})
\label{subclarke3}
\end{equation}
As $\left\{z_k\right\}$ is bounded, then there exists a subsequence denoted also, without loss of generality, by $\left\{z^{k_j}\right\}$ such that $\lim \limits_{j\rightarrow +\infty}z^{k_j}=\bar{z}$, with $\bar{z} \in \mathbb{R}^m_+\backslash \left\{0\right\}$.  From $(\ref{subclarke3})$, we have:
$$
\beta_{k_j} \langle x^{k_j} - x^{{k_j}+1} , \widehat{x} + \lambda d - x^{{k_j}+1} \rangle \leq \langle F(.), z_{k_j} \rangle ^ o (x^{{k_j}+1}, \widehat{x} + \lambda d - x^{{k_j}+1})
$$
Lemma $\ref{algebra}$, $(i)$ and $(ii)$, we have:
$$
\beta_{k_j} \langle x^{k_j} - x^{{k_j}+1} , \widehat{x} + \lambda d - x^{{k_j}+1} \rangle \leq \sum\limits_{i=1}^m z_{k_j}^iF_i^0(x^{{k_j}+1}, \widehat{x} + \lambda d - x^{{k_j}+1}),
$$
where $z_{k_j}^i$ are the components of the vector $z_{k_j}.$ Then using Lemma $\ref{algebra}$, $(iii),$ we obtain:
$$
\beta_{k_j} \langle x^{k_j} - x^{{k_j}+1} , \widehat{x} + \lambda d - x^{{k_j}+1} \rangle \leq \sum\limits_{i=1}^m F_i^0\left(x^{{k_j}+1}, z_{k_j}^i(\widehat{x} + \lambda d - x^{{k_j}+1})\right),
$$
Taking  lim sup in the above inequality, using the condition (\ref{succonsecutivas}), Proposition $\ref{limsup}$ and as $\lambda > 0$, we conclude that
\begin{equation}
0\leq F_1^o(\widehat{x}, d)\bar{z}_1 + ... + F_m^o(\widehat{x}, d)\bar{z}_m
\label{somatorio1}
\end{equation}
Without loss of generality, consider the set $J = \left\{i \in I: \bar{z}_i > 0 \right\}$, where $I = \left\{1,...,m\right\}$.  Thus, from $(\ref{somatorio1})$, there exists $i_0 \in J$ such that
$F_{i_0}^o(\widehat{x}, d)\bar{z}_{i_0} \geq 0$  contradicting $(\ref{desce})$.
\end{proof}


\subsection{Global Convergence}

For this subsection we make the following assumption on the function $F$ and the initial point $x^0$ :
 \begin{enumerate}
\item [${\bf (H_3)}$] The set $\left(F(x^0) - \mathbb{R}^m_+\right)\cap F(\mathbb{R}^n)$ is $\mathbb{R}^m_+$ - complete,  meaning that for all sequences $\left\{a_k\right\}\subset\mathbb{R}^n$, with $a_0 = x^0$, such that $F(a_{k+1}) \preceq F(a_k)$, there exists  $ a \in \mathbb{R}^n$ such that $F(a)\preceq F(a_k), \ \forall \ k \in \mathbb{N}$.
\end{enumerate}
\begin{obs}
The assumption ${\bf (H_3)}$  is cited in various works on proximal point method for convex functions, see Bonnel et al. \cite{Iusem}, Ceng  and Yao \cite{Ceng} and Villacorta and Oliveira \cite{Villacorta}.
\end{obs}
As the  sequence $\left\{x^k\right\}$ generated by {\bf SPPM} algorithm, satisfies the assumption ${\bf (H_3)}$ and from assumptions ${\bf (H_1)}$ and ${\bf (H_2)}$ then 
\begin{center}
$E = \left\{x \in \mathbb{R}^n: F\left(x\right)\preceq F\left(x^k\right),\ \ \forall\  k \in \mathbb{N}\right\}$
\end{center}
 is a nonempty closed convex set.
\begin{proposicao}\textnormal(Fejér convergence)\\
Under assumptions ${\bf(H_1)}$, ${\bf(H_2)}$ and ${\bf(H_3)}$, the sequence $\left\{x^k\right\}$ generated by the SPPM algorithm, $(\ref{inicio})$ and $(\ref{subdiferencial})$, is Fejér convergent to $E$.
\label{fejer2}
\end{proposicao}
\begin{proof} From Theorem $\ref{existe}$, $(\ref{subdiferencial})$,  Lemma $\ref{subcontido}$ and from Lemma $\ref{escalar}$ we obtain that
there exist $g_i^k\in \partial^ oF_i (x^{k+1}),i=1,...,m$ such that
$$
0 \in \sum\limits_{i=1}^{m}z_k^i g_i^k + \alpha_k\left\langle e_k , z_k\right\rangle ( x^{k+1} \ - \ x^k )+ \mathcal{N}_{\Omega_k}(x^{k+1})
$$
where $z_k^i$ are the components of $z_k.$ Thus there exist vectors $g_i^k\in \partial^ oF_i (x^{k+1}),i=1,...,m,$ and $v_k \in \mathcal{N}_{\Omega_k}(x^{k+1})$ such that
\begin{equation}
\sum\limits_{i=1}^{m}z_k^i g_i^k = \beta_k(x^k - x^{k+1}) - v_k
\label{gk}
\end{equation}
where $\beta_k = \alpha_k\left\langle e_k , z_k\right\rangle$, $ \forall \ k \in \mathbb{N}$.  Note that  $\beta_k > 0$, because $\alpha_k > 0$, $e_k$ belongs to $\mathbb{R}^m_{++}$, and $z_k$ belongs to  $\mathbb{R}^m_+\backslash \left\{0\right\}$. From $(\ref{gk})$ we have
\begin{equation}
 x^k - x^{k+1} = \dfrac{1}{\beta_k}\left( \sum\limits_{i=1}^{m}z_k^i g_i^k + v_k \right)
 \label{xk}
 \end{equation}
Now take $x^* \in E,$ then by definition of $E$, $x^* \in \Omega_{k+1}$ for all $k,$ and from Remark \ref{interiornovacio}, there exists 
$\{x^l\}\in \Omega_{k+1}^0$ such that $x^l\rightarrow x^*.$ Observe that, $\forall \ x \in \mathbb{R}^n$:
\begin{eqnarray}
 \left\|x^k - x\right\|^2  =  \left\|x^k - x^{k+1}\right\|^2 + \left\|x^{k+1} - x\right\|^2 +
                              2\left\langle x^k - x^{k+1}, x^{k+1} - x\right\rangle.
\label{norma2}
\end{eqnarray}
 Now,combining  $(\ref{norma2})$, with $x = x^l$, and $(\ref{xk})$, we have:
{\footnotesize
\begin{eqnarray}
\left\|x^k - x^l\right\|^2  & = & \left\|x^k - x^{k+1}\right\|^2 + \left\|x^{k+1} - x^l \right\|^2 +
                               \frac{2}{\beta_k} \left(\sum\limits_{i=1}^{m}z_k^i\left\langle  g_i^k ,\  x^{k+1} -x^l\right\rangle + \left\langle v_k\ ,\  x^{k+1} -x^l\right\rangle\right)
\label{desigualdade2}
\end{eqnarray}
}
As $F(x^l)\prec F(x^{k+1}),$ then $F_i(x^l)<F_i(x^{k+1}), \forall i=1,...,m$.  Furthermore, $g^k_i \in \partial^ o F_i(x^{k+1})$ and as $F_i$ is quasiconvex, using Proposition $\ref{t11}$ we have
\begin{equation}
\left\langle g_i^k ,\  x^{k+1} -x^l\right\rangle \geq 0, \forall i=1,...,m.
\label{conclusao}
\end{equation}
Now, as $v_k \in \mathcal{N}_{\Omega_k}(x^{k+1})$, the inequality $(\ref{desigualdade2})$ and $(\ref{conclusao})$, imply, taking $l\rightarrow \infty$
\begin{eqnarray}
 0\leq \left\|x^{k+1} - x^k\right\|^2 \leq \left\|x^k - x^*\right\|^2 - \left\|x^{k+1} - x^*\right\|^2, \forall k \in \mathbb{N}
 \label{desigual}
 \end{eqnarray}
 Thus,
\begin{equation}
\left\|x^{k+1} - x^*\right\| \leq \left\|x^k - x^*\right\|
\label{fejer3}
\end{equation}
\end{proof}
\begin{proposicao}
Under  assumptions ${\bf(H_1)}$, ${\bf(H_2)}$ and ${\bf(H_3)},$ the sequence   $\left\{x^k\right\}$ generates  by the {\bf SPPM} algorithm, $(\ref{inicio})$ and $(\ref{subdiferencial}),$ satisfies
\begin{center}
$\lim  \limits_{k\rightarrow +\infty}\left\|x^{k+1} - x^k\right\| = 0$.
\label{decrescente1}
\end{center}
\end{proposicao}

\begin{proof}
It follows from $(\ref{fejer3})$, that $ \forall x^* \in E$,  $\left\{\left\|x^k - x^*\right\|\right\}$ is a nonnegative and nonincreasing sequence, and hence is convergent.  Thus, the right-hand side of $(\ref{desigual})$ converges to 0 as $k \rightarrow +\infty$, and the result is obtained.
\end{proof}

\begin{proposicao}
Under assumptions ${\bf(H_1)}$, ${\bf(H_2)}$ and ${\bf(H_3)},$ the sequence  $\left\{x^k\right\}$  generated by the {\bf SPPM} algorithm converges some point of $E$.
\label{acumulacao1}
\end{proposicao}
\begin{proof}
From Proposition $\ref{fejer2}$ and Lemma $\ref{fejerlim1}$, $(i)$, $\left\{x^k\right\}$ is bounded, then exists a subsequence $\left\{x^{k_j}\right\}$ such that $\lim  \limits_{j\rightarrow +\infty}x^{k_j} = \widehat{x}$. Since $F$ is locally Lipschitz on $\mathbb{R}^n$, then the function  $\left\langle F(.), z\right\rangle$ is also locally Lipschitz on $\mathbb{R}^n$ and so, continuous for all $z \in \mathbb{R}^m$, in particular, for all $z \in \mathbb{R}^m_+ \backslash \left\{0\right\}$ , and
$\lim \limits_{j\rightarrow +\infty}\left\langle F(x^{k_j}) , z\right\rangle = \left\langle F(\widehat x),z\right\rangle$.  On the other hand, as $x^{k+1} \in \Omega_k$, we have $F(x^{k+1}) \preceq F(x^{k})$ and since $z \in \mathbb{R}^m_+ \backslash \left\{0\right\}$, we conclude $\left\langle F(x^{k+1}) , z\right\rangle \leq \left\langle F(x^{k}) , z\right\rangle $. Furthermore, from Remark $\ref{limitada}$, we can assume  that the function $\left\langle F(.), z\right\rangle$ is bounded below, for each $z \in \mathbb{R}^m_+\backslash \left\{0\right\}$.  Then the sequence $\left\{\left\langle F(x^k),z\right\rangle\right\}$ is nonincreasing and bounded below, hence convergent. So {\small
$\lim \limits_{k\rightarrow +\infty}\left\langle F(x^{k}) , z\right\rangle = \left\langle F(\widehat {x}),z\right\rangle = inf_{k\in \mathbb{N}}\left\{\left\langle F(x^k),z\right\rangle\right\} \leq  \left\langle F(x^k),z\right\rangle$}. Thus,
$\left\langle F(x^k)-F(\widehat{x}),z\right\rangle \geq  0, \forall \ k \in \mathbb{N}, \forall \ z \in \mathbb{R}^m_+ \backslash \left\{0\right\}$. We conclude that $F(x^k) - F(\widehat{x}) \in \mathbb{R}^m_+$, i.e, $F(\widehat{x})\preceq F(x^k), \forall \ k  \in \mathbb{N}$. Thus $\widehat{x}\in E,$ then using Lemma $\ref{fejerlim1}$, $(ii)$, we obtain the result.
\end{proof}\\
Finally, we prove  that the sequence of the iterations converges to a Pareto-Clarke critical point when the sequence of regularization parameters $\lbrace\alpha_k\rbrace$ is bounded.

\begin{teorema}
Consider $F:\mathbb{R}^n\longrightarrow \mathbb{R}^m $ a function satisfying the assumptions $\bf (H_1)$, $\bf (H_2)$ and $\bf (H_3)$.  If $0 < \alpha_k < \tilde{\alpha}$, then the sequence $\lbrace x_k\rbrace$  generated by the SPPM algorithm, $(\ref{inicio})$ and $(\ref{subdiferencial})$, converges to a Pareto-Clarke critical point of the problem $(\ref{pom})$.
\end{teorema}
\begin{proof} From Proposition \ref{acumulacao1}, $\{x^k\}$ converges, then this sequence has a unique cluster point $\bar x$ and from Proposition \ref{decrescente1} and Proposition \ref{prop1} we obtain the result.
\end{proof}
\begin{corolario}
If $F:\mathbb{R}^n\longrightarrow \mathbb{R}^m $ is $\mathbb{R}^m_{+}$-convex and $\bar{x}$  the point of convergence given by the SPPM algorithm, given by $(\ref{inicio})$ and $(\ref{subdiferencial})$,  then $\bar{x}$ is weak Pareto solution of the problem $(\ref{pom})$.
\end{corolario}

\begin{proof}
It is inmediate from Proposition \ref{Corconvexlips}.
\end{proof}
\begin{corolario}
If $F:\mathbb{R}^n\longrightarrow \mathbb{R}^m $ is continuously differentiable on $\mathbb{R}^n$ and satisfies the assumptions $\bf (H_2),$ $\bf (H_3),$ then the point of convergence given by the SPPM algorithm $\bar{x}$  , given by $(\ref{inicio})$ and $(\ref{subdiferencial})$, is a Pareto critical point of the problem $(\ref{pom}),$ that is, there exists $i_0\in \{1,2,...,m\}$ such that
$$
\left\langle \nabla F_{i_0}(\bar x),d\right\rangle\geq 0, \forall d\in \mathbb{R}^n.  
$$
\end{corolario}
\begin{proof}
It is immediate since continuously differentiable on $\mathbb{R}^n$ implies the assumption $\bf (H_1),$ and $=F_i^0(x,d)=F_i'(x,d)=\left\langle \nabla F_i(x),d\right\rangle,$ where $F_i'$ is the directional derivative of $F_i.$ 
\end{proof}
\section{Optimization models with quasiconvex multivalued functions}
\label{modelos}
In this section we present some general quasiconvex multiobjective problems where the proposed algorithm may be applied.
\subsection{A quasiconvex model in demand theory}
\noindent

Let \textbf {n} be a finite number of consumer goods. A consumer is an agent who must choose how much to consume of each good.  An ordered set of numbers representing the amounts consumed of each good set is called vector of consumption, and denoted by $ x =  (x_1, x_2, ...,  x_n) $ where $ x_i $ with $ i = 1,2,  ..., n $,  is the quantity consumed  of good \textbf {i}. Denote by $ X $, the feasible set of these vectors  which will be called the  set of consumption,  usually in economic applications we have $ X  \subset \mathbb{R} ^ n_ + $.

In the classical approach of demand theory, the analysis of consumer behavior starts specifying a preference relation over the set $X,$ denoted by $\succeq$. The notation: $  "x \succeq  y "  $ means that "$ x $ is at least as good as  $ y $" or "$  y $ is not preferred to $x$". This preference relation  $ \succeq  $ is assumed rational, i.e,  is complete because the consumer is able to order all possible combinations of goods, and  transitive, because consumer preferences are consistent, which means if the consumer prefers $\bar{x}$ to $\bar{y} $ and $\bar{y}$ to $\bar{z}$,  then he prefers $\bar{x}$ to $\bar{z} $ (see Definition  3.B.1  of Mas-Colell et al. \cite{Colell}). 

A function  $ \mu:X \longrightarrow \mathbb {R} $ is said to be an utility function representing a preference relation $\succeq $ on $X$, if the following condition is satisfied:
\begin{equation}
x \succeq y , \textrm{if and only if,}\  \mu(x) \geq \mu(y)
\label{prefe}
\end{equation}
for all $x, y \in X$.

The utility function is a way to represent preferences between  two vectors of consumption. If they have  the same value of the utility function, then the consumer is indifferent. Moreover, if we have  several preferences relations $\succeq_i, i=1,2,...,m,$  (multiple  criteria), which satisfy the condition $(\ref{prefe})$, then we have a utility function $\mu_i$ for each one of these preferences $\succeq_i$.

Observe that the utility function not always exist. In fact, define in $  X = \mathbb {R} ^ 2 $  a lexicographic relation,  given by: for $ x, y \in \mathbb {R}^2 $, $ x \succeq y $ if and only if $ "x_1 > y_1"$ or $"x_1 = y_1 \ \textnormal {e} \ x_2 \geq y_2 "$. Fortunately, a very general class of preference relations can be represented by utility functions, see for example 3.C.1 Proposition of Mas-Colell et al. \cite {Colell}.

If a preference relation $ \succeq $  is represented by a utility function $ \mu $, then the problem of maximizer the consumer preference on  $X$ is equivalent to solve the optimization problem
\begin{eqnarray*}
\textnormal{(P)\ max}\{ \mu(x) :x \in X\}.
\end{eqnarray*}

Now consider a multiple criteria, that is, consider  $ m $ preference relations denoted by $\succeq_i, i=1,2,...,m.$ Suppose that for each preference $\succeq_i,$  there exists an utility  function, $ \mu_i,$ respectively, then the problem of maximizer the consumer preference on  $ X $ is equivalent to solve the multiobjective optimization problem
\begin{eqnarray*}
\textnormal{(P')\ max}\{ (\mu_{1}(x), \mu_{2}(x), ..., \mu_{m}(x)) \in \mathbb{R}^m :x \in X\}.
\end{eqnarray*}

Since there is not a single point which maximize all the functions simultaneously the concept of optimality is established in terms of Pareto optimality or efficiency.

On the other hand, a natural psychological assumption in economy is that the consumer tends to diversify his consumption among all goods, that is, the preference $ \succeq $  satisfies the following  convexity property: $ X$ is convex and if $x \succeq  z $ and $ y \succeq z $ then $ \lambda  x + (1 - \lambda) y \succeq  z $, $ \forall  \lambda \in [0,1] $. 

It can be proved that  if there is a utility function representing the preference relation $\succeq,$ then the convexity property of $ \succeq$ is equivalent to the quasiconcavity of the utility function $ \mu $.  Therefore ${\rm(P')} $ becomes a maximization problem with quasiconcave multiobjective function, since each component function is quasiconcave. 

Taking F = $ (- \mu_1, - \mu_2, ..., - \mu_m) $, we obtain a minimization problem with quasiconvex multiobjective function, since each  component function is quasiconvex one.

There are various class of utilities functions which are frequently used  to generate demand functions.  One of the most common  is the Cobb-Douglas utility function,  which is defined on $ \mathbb{R} ^ 2$  by $ \mu(x_1, x_2) = k x_1 ^\alpha x_2 ^ {\beta} $, with $ \alpha,\beta>0$ and $ k > 0 $. Another utility function CES (Constant Elasticity of Substitution), defined on $ \mathbb {R} ^ 2$ by $ \mu (x_1, x_2) = (\lambda_1x_1 ^ \rho + \lambda_2x_2 ^ \rho) ^ {1 / \rho}$, where $ \lambda_1, \lambda_2 \geq 0$, $ \lambda_1 + \lambda_2 = 1 $, and $ \rho $ is a constant.

\subsection{A quasiconvex model in location theory}
\noindent

Location problems are related to determining the location for one or more facilities, considering a given set of demand points, with which interactions should be established. These terms are not part of a standard terminology, are sometimes replaced by: clients, existing facilities,  businesses or users.

The following problem of locating a facility is motivated from the Chapter IV of Gromicho, \cite{Gromicho}. For each $ i = 1 , ... , m,$ let the cluster set $ d^i = \{ d ^{i}_1 , d^{i}_2,..., d ^{i}_{p(i)} \} \subset \mathbb { R} ^n $ , $ n \geq $ 2 (there exist $m$ cluster). We need to find a location $ x \in \mathbb {R} ^ n $ for an installation so that this location minimizes some real function involving the distance between the new location and each cluster set of demand points.

For each $ i = 1 , ... , m $ , if $ C ^ i_j $ , $ j = 1 , ... , p(i),$ are compact convex sets with $ 0 \in \textnormal {int } ( C ^ i_j ) $  and $ \textnormal { int } ( C ^ i_j ) $ denotes the interior of $ C ^ i_j $ then, for each $ i = 1 , ... , m,$ we define the distance between $ x $ and $ d ^ i_j$ by $ \gamma_{C ^{i}_j} ( x - d ^ i_j ) $ with $ \gamma_ {C^ i_j}$ the gauge or  Minkowsky functional of the set $C^i_j$, i.e. $ \gamma_ {C^i_j}(x) = \textnormal {inf} \{ t > 0: x \in tC ^i_j  \}$ . Note that if $ { C ^ i_j } $ is the unit ball in $ \mathbb {R} ^ n,$ then $ \gamma_ { C ^ i_j } ( x ) $ is the Euclidean distance from $ x $ to $ 0 $.

To introduce the model, consider, for each $ i = 1,  ..., m $, the function  $\gamma_i:  \mathbb{R} ^ n \longrightarrow \mathbb {R}^p_+$,  given by $ \gamma_i (x) = (\gamma_{C^ i_1}(x -  d ^i_1), ..., 
\gamma_{C ^i_{p(i)}} (x  - d^i_{p(i)})) $. And suppose, for each $i,$ that the functions $ f^i_j: \mathbb {R} ^{p(i)}_ + \longrightarrow 
\mathbb {R}_ +$, with $ j  = 1, ..., p(i) $ is nondecreasing  in $ \mathbb {R}^ {p(i)}_ + $,  that is, if $ x, y \in \mathbb{R} ^{p(i)}_+$, satisfying for each $ j = 1, ..., p(i)$, $ x_j \leq y_j $,  then $f^i_j (x) \leq f^i_j (y) $.

The localization model is given by
\begin {center}
$min \{(\phi_1 (x), \phi_2 (x), ..., \phi_m (x)): x \in \mathbb {R} ^ n \} $,
\end {center}
where, for each $ i = 1, ..., m$,  $\phi_i (x) = max_{1  \leq j \leq p(i)} f^i_j (\gamma_i (x)) $. If for each $i=1,...,m,$ the functions $ f_j^i:  \mathbb {R} ^{p(i)}_+ \longrightarrow \mathbb {R} ^ + $  are quasiconvex in $ \mathbb {R}  ^{p(i)}_ + $, then it can proved that for every $ i = 1,  ..., m $, each function  $ \phi_i (.) $ is quasiconvex  in $ \mathbb {R} ^ n $.


\section{Conclusion and future works}
\noindent

In this paper we introduced a scalarization proximal point method to solve unconstrained (possibly nonconvex and non-differentiable) multiobjective minimization problems with locally Lipschitz functions. Then, for quasiconvex objective functions we show a strong convergence (global convergence)  to a Pareto-Clarke critical point satisfying the completeness assumption
$\bf (H_3)$. Note this assumption has been considered in the convergence analysis of the proximal point method for the convex case, see \cite{Bello}.

We also present, in Section \ref{modelos}, two optimization models where the quasiconvexity of the multiobjective functions appear naturally. We present quasiconvex models in demand theory and location theory.  

The (SPPM) algorithm, introduced in this paper, is the first attempt to construct efficient proximal point methods to solve quasiconvex multiobjective minimization problems and in its actual version may be considered as a based algorithm to develop other methods that consider computational errors, lower computational costs, lower complexity order and improves the convergence rate. Observe that in this paper we do not present an inexact version because, according to our knowledge,  the theory of $\epsilon-$ subdifferencial  Clarke has not yet been developed.

To reduce considerably the computational cost in each iteration of the (SPPM) algorithm it is need to consider the unconstrained iteration
\begin{equation}
\label{subdiferencialint}
 0 \in \partial^o\left( \left\langle F(.), z_k\right\rangle  + \dfrac{\alpha_k}{2}\left\langle e_k , z_k\right\rangle \Vert\ .\  - x^k \Vert ^2 \right) (x^{k+1})
\end{equation}
which is more practical that (\ref{subdiferencial}). One natural condition to obtain (\ref{subdiferencialint}) is that $x ^{k +1} \in (\Omega_k)^0$ (interior of $\Omega_k$). So we believe that a variant of the (SPPM) algorithm may be an interior variable metric proximal point method.

Observe also that in practice the iteration (\ref{subdiferencial}) or (\ref{subdiferencialint}) should be solve using a local algorithm, which only provides an approximate solution. Therefore, we consider that in a future work it is important to analyze the convergence of the proposed algorithm considering now inexact iterations, see \cite{PapaLenninOliveira}.  Also the introduction of bundle methods are welcome.\\



{\noindent}{\bf Acknowledgements}\\
The research
of H.C.F.Apolinário was partially supported by CAPES/Brazil. The research of P.R.Oliveira was
partially supported by CNPQ/Brazil. The research of E.A.Papa Quiroz was partially supported by
the Postdoctoral Scholarship CAPES-FAPERJ Edital PAPD-2011.



\end{document}